\documentclass{amsart}
\usepackage{tikz}
\usepackage{amsmath, amssymb}
\input xy
\xyoption{all}
\usepackage{xcolor, hyperref}

\newif\ifhascomments \hascommentstrue
\ifhascomments
  \newcommand{\david}[1]{{\color{red}[[\ensuremath{\bigstar\bigstar\bigstar} #1]]}}
  \newcommand{\matt}[1]{{\color{red}[[\ensuremath{\spadesuit\spadesuit\spadesuit} #1]]}}
\else
  \newcommand{\david}[1]{}
  \newcommand{\matt}[1]{}
\fi

\newtheorem{theorem}{Theorem}[section]
\newtheorem{lemma}[theorem]{Lemma}

\newtheorem{fact}[theorem]{Fact}
\newtheorem{proposition}[theorem]{Proposition}

\theoremstyle{definition}

\newtheorem{remark}[theorem]{Remark}
\newtheorem{definition}[theorem]{Definition}

\begin{document}

\title{On the Weierstrass Preparation Theorem over General Rings}

\author{Jason Bell}
\address{Department of Pure Mathematics\\ University of Waterloo\\
Waterloo, Ontario\\
N2L 3G1}
\email{jpbell@uwaterloo.ca}

\author{Peter Malcolmson}
\address{Department of Mathematics\\ Wayne State University \\ 656 W. Kirby St.\\
 Detroit, MI, USA \\
 48202}
\email{petem@wayne.edu}

\author{Frank Okoh}
\address{Department of Mathematics\\ Wayne State University \\ 656 W. Kirby St.\\
 Detroit, MI, USA \\
 48202}
\email{okoh@wayne.edu}

\author{Yatin Patel}
\address{Department of Mathematics\\ Wayne State University \\ 656 W. Kirby St.\\
 Detroit, MI, USA \\
 48202}
\email{yatin@wayne.edu}

\begin{abstract}
We study rings over which an analogue of the Weierstrass preparation theorem holds for power series. We show that a commutative ring \( R \) admits a factorization of every power series in \( R[[x]] \) as the product of a polynomial and a unit if and only if \( R \) is isomorphic to a finite product of complete local principal ideal rings. We also characterize Noetherian rings \( R \) for which this factorization holds under the weaker condition that the coefficients of the series generate the unit ideal: this occurs precisely when $R$ is isomorphic to a finite product of complete local Noetherian integral domains.

Beyond this, we investigate the failure of Weierstrass-type preparation in finitely generated rings and prove a general transcendence result for zeros of $p$-adic power series, producing a large class of power series over number rings that cannot be written as a polynomial times a unit. Finally, we show that for a finitely generated infinite commutative ring $R$, the decision problem of determining whether an integer power series (with computable coefficients) factors as a polynomial times a unit in \( R[[x]] \) is undecidable.
\end{abstract}
\maketitle

\section{Introduction}
The Weierstrass preparation theorem is a tool of fundamental importance in non-Archimedean analysis, which states that for a complete local ring $R$, if a power series $f(x)=a_0+a_1x+\cdots \in R[[x]]$ has the property that, for some $n$, $a_0,\ldots ,a_n$ generate the unit ideal, then there exists a polynomial $p(x)\in R[x]$ of degree $n$ and a unit of the power series ring $u(x)$ such that $f(x)=p(x)u(x)$.  This result plays an important role within algebraic and analytic geometry, analytic number theory, Iwasawa theory, the study of central simple algebras, model theory, and other areas (see, for example, \cite{Serre, BGR, BurnsGreither, VenjakobWeierstrass, HarbaterHartmannKrashen, CluckersLipshitz, vddMiller, Denefvdd}) as it often allows one to understand ideals arising in analytic settings in terms of polynomial relations. 

The main purpose of this paper is to study the class of rings $R$ for which one has a Weierstrass preparation theorem as above.  In particular, there are two natural variants to consider: a strong version (which we call the \emph{strong Weierstrass property}) in which one insists that \emph{every} power series can be factored as a product of a polynomial and a unit, and a weaker version---the \emph{Weierstrass property}---in which only power series whose coefficients generate the unit ideal are required to factor as in the statement of the Weierstrass preparation theorem (see Definition \ref{defn:Weierstrass} for a precise definition). We note that this ``weaker version'' is more in the spirit of the classical Weierstrass preparation theorem, and so we have not added an adjective in this case.  Moreover, while the stronger version is not \emph{a priori} stronger when one carefully examines Definition \ref{defn:Weierstrass}, we shall nevertheless see that the strong Weierstrass property implies the ordinary property (see Remark \ref{rem:strong}).

The stronger variant of the Weierstrass preparation theorem does not generally hold in the complete local case, but it does hold in many of the most important settings in which one uses Weierstrass preparation; namely the case when the coefficient ring $R$ is a ring of $p$-adic integers or a one-variable power series ring.  For this reason we isolate this stronger variant of the property and investigate the settings in which it holds. Our first main result shows that every ring satisfying the strong Weierstrass property is, in a structural sense, built from rings exhibiting similar behavior to the examples given above.

\begin{theorem}\label{thm:main1}
Let \( R \) be a commutative ring. Then \( R \) has the strong Weierstrass property if and only if \( R \) is a complete principal ideal ring.  
In particular, \( R \) has the strong Weierstrass property precisely when there exists a natural number \( m \) and rings \( R_1, \ldots, R_m \) such that
\[
R \cong R_1 \times \cdots \times R_m,
\]
where for each \( i \), the ring \( R_i \) is either a complete discrete valuation ring or an Artinian local principal ideal ring.
\end{theorem}

Complete discrete valuation rings (DVRs) are well understood and have a concrete classification. If \( R \) is a complete DVR with maximal ideal \( P \) and residue field \( k \), then there are two main cases:

\begin{itemize}
    \item \textbf{Equicharacteristic case}: \( \mathrm{char}(R) = \mathrm{char}(k) \). In this situation, Cohen's structure theorem implies that \( R \cong k[[t]] \), the formal power series ring over the field \( k \); see Cohen~\cite{Cohen}.
    
    \item \textbf{Mixed characteristic case}: \( \mathrm{char}(R) = 0 \) and \( \mathrm{char}(k) = p > 0 \). This case is more subtle but still well understood. When the residue field \( k \) is perfect and the maximal ideal of \( R \) is generated by \( p \), the ring \( R \) is isomorphic to the ring of Witt vectors \( W(k) \), which generalizes the construction of the \( p \)-adic integers \( \mathbb{Z}_p \), corresponding to \( k = \mathbb{F}_p \), see Serre~\cite{Serre} and Hazewinkel~\cite{Hazewinkel} for details.
\end{itemize}

On the other hand, \textbf{Artinian local principal ideal rings} are particularly well-behaved. They have a unique maximal ideal \( P \), and there exists some integer \( n \geq 1 \) such that \( P^n = (0) \). The only proper ideals of the ring are \( P, P^2, \ldots, P^n \). These rings can in some sense be viewed as ``truncated DVRs.'' 

The usual Weierstrass property (that is, having a true analogue of the Weierstrass preparation theorem) is somewhat more subtle and requires more detailed arguments, even in the Noetherian case.  We know in this case that all complete local rings have this property and we are again able to give a classification of rings with this property. 
\begin{theorem}
\label{thm:main2} Let $R$ be a commutative Noetherian ring. Then $R$ has the Weierstrass property precisely when $R$ is isomorphic to a finite product of complete local Noetherian rings.
\end{theorem}
For both the Weierstrass property and its stronger variant, any non-Artinian ring satisfying either condition admits a surjective map onto a complete local integral domain of positive Krull dimension. In particular, in the non-Artinian case such a ring must be uncountable.  It is thus interesting to study the extent to which the Weierstrass property fails for non-complete, non-Artinian rings.  A natural starting point is the ring of integers, or more generally, the ring of integers in a number field.  Here one encounters thorny problems in Diophantine approximation; namely, questions of whether power series over the ring of $p$-adic integers have algebraic zeros.  

We are able to give a general transcendence criterion for integer power series, which shows as a corollary that such power series are never unit multiples of a polynomial.  Here, we recall that given a prime number $p$, we can regard the integers as a dense subring of the complete local ring $\mathbb{Z}_p$ of $p$-adic integers, which is a metric space with norm $|\, \cdot \,|_p$.  Then an integer power series $f(x)$ can be regarded as a function that is analytic in the open unit disc of $\mathbb{Z}_p$ (the set of all numbers of $p$-adic absolute value $<1$) and more generally of the open unit disc of $\mathbb{C}_p$, which is the field obtained by taking the completion of the algebraic closure of the field of fractions, $\mathbb{Q}_p$, of the ring of $p$-adic integers.  The field $\mathbb{C}_p$ is a $p$-adic analogue of the complex numbers and as such, just as we can in the case of complex numbers, we have elements of $\mathbb{C}_p$ that are \emph{algebraic} (roots of nonzero integer polynomials) and elements that are \emph{transcendental} (i.e., elements that are \emph{not} algebraic over the subfield of rational numbers).  We note that the $p$-adic absolute value extends to an absolute value on $\mathbb{C}_p$ and that $\mathbb{C}_p$  is algebraically closed. We refer the reader to the book of Gouvêa \cite{Gouvea} for further details.
\begin{theorem}
\label{thm:main3}
Let $p$ be a prime number and let $\{b(n)\}_{n\ge 0}$ be a strictly increasing sequence of nonnegative integers with $b(0)=1$ such that $b(n)$ grows super-exponentially as $n\to \infty$. Suppose in addition that 
$a_0, a_1, a_2,\ldots $ are nonzero integers such that:
\begin{enumerate}
\item $p\mid a_0$;
\item there is some $i\ge 1$ such that $p\nmid a_i$;
\item there are positive constants $C,\kappa >1$ such that $|a_n|\leq \kappa\cdot C^{b(n)}$ for all $n$.
\end{enumerate}
Then the power series $$f(x):=\sum_{n\ge 0} a_n x^{b(n)}$$ has a root $\lambda$ in $\{z\in \overline{\mathbb{Q}}_p\colon |z|_p<1\}$ that is transcendental over $\mathbb{Q}$, and $f(x)$ cannot be factored in $\mathbb{Z}[[x]]$ as a product of an integer polynomial and a unit of $\mathbb{Z}[[x]]$, which demonstrates the failure of the Weierstrass property in this setting.
\end{theorem}
We also prove a transcendence result of a similar flavor in positive characteristic (see Theorems \ref{thm:transcendencep} and \ref{rem:nonassociate}).  

Finally, we are able to use Theorem \ref{thm:main3} to prove the following general undecidability result.
\begin{theorem} Let $T$ be an infinite finitely generated commutative ring.
The problem of determining whether a computable power series \( f(x) \in T[[x]] \) is expressible as the product of a unit in $T[[x]]$ and a polynomial in $T[x]$ is undecidable.
\label{thm:undecidable}
\end{theorem}
This says there exists no general algorithm which takes as input a power series over $R$ with coefficients that can be computed with a Turing machine and which then decides whether such a factorization exists.  (See \S\ref{undecidable} for background and Theorem \ref{thm:undecidable} for the formal statement.) 

The outline of this paper is as follows.  In \S\ref{SWP} we give some short background on complete local rings and prove Theorem \ref{thm:main1}. In \S\ref{WP} we then prove Theorem \ref{thm:main2}. In \S\ref{countable} we prove that a dichotomy holds: for a countable ring $R$, either $R$ has the strong Weierstrass property or there are uncountably many associate classes in $R[[x]]$ that do not contain a polynomial. Then in \S\ref{transcendence} we prove a general transcendence result about zeros of $p$-adic power series, as well as a positive characteristic analogue of this result and use this to prove Theorem \ref{thm:main3} (see also Theorem \ref{rem:nonassociate}). Finally, in \S\ref{undecidable} we give some background on decision procedures and prove Theorem \ref{thm:undecidable}.  
\subsection{Notation}
Unless stated otherwise, all rings in this paper are commutative with a nonzero identity. The multiplicative group of units of $R$ is denoted $R^{\times}$ and the {\it associate class} of an element $r \in R$ is the set $R^{\times}r := \{ur: u \in R^{\times}\}$. For $r,r'\in R$, we say that $r$ is \emph{associate} to $r'$ in $R$ if $r\in r' R^{\times}$.  Given an integral domain $R$, we let ${\rm Frac}(R)$ denote its field of fractions.
\section{The strong Weierstrass property}
 \label{SWP}
 We recall that if \( I \) is an ideal of a ring \( R \), then we can form the \( I \)-\emph{adic completion} of \( R \), denoted \( \widehat{R} \), as the subring of the direct product \( \prod_{n \geq 1} R/I^n \) consisting of sequences \( (r_n + I^n)_{n \geq 1} \) such that for all \( n \geq 2 \), we have \( r_n - r_{n-1} \in I^{n-1} \). This ring inherits a natural topology, called the \( I \)-\emph{adic topology}, where a basis of open neighborhoods of an element \( \mathbf{r} = (r_n + I^n)_{n \geq 1} \) is given by the sets
\[
U_m(\mathbf{r}) := \left\{ (s_n + I^n)_{n \geq 1} \in \widehat{R} : s_n - r_n \in I^m \text{ for all } n \geq 1 \right\},
\]
where $m$ ranges over all positive integers.

The ring \( \widehat{R} \) is complete with respect to this topology; that is, all Cauchy sequences in \( \widehat{R} \) converge. Moreover, there is a natural homomorphism \( \phi : R \to \widehat{R} \) given by \( \phi(r) = (r + I^n)_{n \geq 1} \), whose kernel is \( \bigcap_{n \geq 1} I^n \). In particular, \( \phi \) is injective when this intersection is zero, and the image of \( R \) is a dense subring of \( \widehat{R} \). A local ring \( R \) is called \emph{complete} if \( R \to \widehat{R} \) is an isomorphism when we complete \( R \) with respect to its unique maximal ideal.

We refer the reader to \cite[Chapter 7]{Eisenbud} for more details about completions of rings.

\vspace{0.5em}

We now recall the statement of the Weierstrass Preparation Theorem. A proof can be found in Lang \cite[Theorem 9.2]{Lang}.

\begin{theorem}[Weierstrass Preparation Theorem]\label{thm:WPT}
Let \( R \) be a complete local ring with maximal ideal \( P \), and let
\[
f(x) = \sum_{i=0}^\infty a_i x^i \in R[[x]]
\]
be a power series such that \( a_0, \ldots, a_{n-1} \in P \) and \( a_n \notin P \). Then \( f(x) \) can be expressed uniquely as the product of a monic polynomial of degree \( n \) in \( R[x] \) and a unit in \( R[[x]] \).
\end{theorem}
This result allows one to reduce the study of analytic objects to algebraic ones in many settings. This is a core principle behind rigid analytic geometry.
We note that the unique monic polynomial we obtain in the statement of the Weierstrass Preparation Theorem is often called a \emph{Weierstrass polynomial}.

In light of this result, we give the following definition.
\begin{definition} \label{defn:Weierstrass} Let $R$ be a ring.  We say that $R$ has the \emph{strong Weierstrass property} if every power series over $R$ factors as a product of a polynomial and a unit of $R[[x]]$; we say that $R$ has the \emph{Weierstrass property} if, whenever $f(x)=a_0+a_1x+\cdots  \in R[[x]]$ has the property that $a_0,\ldots ,a_n$ generate the unit ideal, there is a monic polynomial $p(x)\in R[x]$ of degree $n$ and a unit $u(x)$ of $R[[x]]$ such that $f(x)=p(x)u(x)$.
\end{definition}
We note that the stronger variant is not \emph{a priori} stronger in a technical sense, but only in the sense that it provides a factorization of \emph{all} power series as the product of a polynomial and a unit.   Nevertheless, Theorems \ref{thm:main1} and \ref{thm:main2} show that the strong variant implies the ordinary variant (see Remark \ref{rem:strong}).
\vspace{0.5em}

Unfortunately, the Weierstrass Preparation Theorem does not guarantee that \emph{every} power series over a complete local ring is associate to a polynomial. However, in the case that \( R \) is a complete local ring whose maximal ideal is principal, this property follows quickly.

\begin{lemma}\label{CLPIPSAP}
Let \( R \) be a complete local ring with maximal ideal \( P \). If \( P \) is principal, then every power series in \( R[[x]] \) is associate to a polynomial.
\end{lemma}

\begin{proof}
Let \( f(x) = a_0 + a_1 x + a_2 x^2 + \cdots \in R[[x]] \) be a nonzero power series, and let \( I \) denote the ideal generated by the coefficients \( a_0, a_1, \ldots \). Then since $R$ is complete, the intersection of the powers of $P$ is zero and so by \cite[Chapter VI, \S 1.4, Proposition 2]{Bourbaki}, \( I = (b^n) \) for some \( n \geq 0 \), where \( b \) is a generator for \( P \). Thus we can write \( f(x) = b^n g(x) \), where \( g(x) \in R[[x]] \) has at least one coefficient not in \( P \). Then \( g(x) \) is associate to a Weierstrass polynomial by the Weierstrass Preparation Theorem, so we can write \( g(x) = p(x) u(x) \), where \( p(x) \in R[x] \) and \( u(x) \in R[[x]]^{\times} \). It follows that
\[
f(x) = b^n p(x) u(x),
\]
which expresses \( f(x) \) as a product of a polynomial and a unit.
\end{proof}

The next result shows that the class of rings with the strong Weierstrass property has certain closure properties, and will be useful in obtaining Theorem \ref{thm:main1}.

\begin{lemma}\label{prop:1.3}
The class of rings with the strong Weierstrass property is closed under homomorphic images and formation of finite direct products.
\end{lemma}

\begin{proof}
Suppose \( R \) has the strong Weierstrass property and let \( I \) be a proper ideal of \( R \). Let \( f(x) \in (R/I)[[x]] \). We may lift \( f(x) \) to a power series \( g(x) \in R[[x]] \) such that \( g(x) \mod I = f(x) \). By assumption, \( g(x) = p(x) u(x) \) with \( p(x) \in R[x] \) and \( u(x) \in R[[x]]^{\times} \). Reducing modulo \( I \), we obtain a factorization \( f(x) = \overline{p(x)} \cdot \overline{u(x)} \), and since polynomials and units descend under the canonical projection \( R[[x]] \to (R/I)[[x]] \), it follows that \( R/I \) also has the strong Weierstrass property.

To show closure under finite direct products, suppose \( R_1 \) and \( R_2 \) have the strong Weierstrass property. Then so does \( R_1 \times R_2 \), since
\[
(R_1 \times R_2)[[x]] \cong R_1[[x]] \times R_2[[x]].
\]
Given \( (f_1(x), f_2(x)) \in R_1[[x]] \times R_2[[x]] \), by assumption there exist polynomials \( p_i(x) \in R_i[x] \) and units \( u_i(x) \in R_i[[x]]^{\times} \) such that \( f_i(x) = p_i(x) u_i(x) \) for \( i = 1, 2 \). Thus,
\[
(f_1(x), f_2(x)) = (p_1(x), p_2(x)) \cdot (u_1(x), u_2(x)),
\]
which is a factorization into a polynomial and a unit in \( (R_1 \times R_2)[[x]] \).
\end{proof}
 
 We now investigate the associate classes of a power series ring and show that the number of classes that do not contain a polynomial representative is uncountable under general conditions.

\begin{lemma}
Let \( R \) be a ring that is not Noetherian. Then \( R \) has an uncountable set of pairwise non-associate power series, none of which are associate to a polynomial.
\label{lem:noether}
\end{lemma}

\begin{proof}
Let \( a_0, a_1, \ldots \in R \) be a sequence such that \( a_n \notin (a_0, \ldots, a_{n-1}) \) for all \( n \geq 1 \). We note that such a sequence necessarily exists, since $R$ is not Noetherian. For each sequence \( \mathbf{e} = \{ \epsilon_i \}_{i \geq 0} \in \{0,1\}^\mathbb{N} \) such that \( \epsilon_i = 1 \) for infinitely many \( i \), define the power series
\[
f_{\mathbf{e}}(x) = \sum_{n \geq 0} a_n \epsilon_n x^n.
\]
We claim that for distinct sequences \( \mathbf{e} \neq \mathbf{e}' \) satisfying the above condition, the power series \( f_{\mathbf{e}}(x) \) and \( f_{\mathbf{e}'}(x) \) are not associate, and neither is associate to a polynomial.

To prove the first claim, suppose \( \epsilon_n = 1 \), \( \epsilon_n' = 0 \) for some \( n \). Let \( u(x) = u_0 + u_1 x + \cdots \in R[[x]]^\times \). The coefficient of \( x^n \) in \( f_{\mathbf{e}}(x) u(x) \) is
\[
a_n u_0 + a_{n-1} u_1 + \cdots + a_0 u_n.
\]
Since \( u_0 \) is a unit and \( a_n \notin (a_0, \ldots, a_{n-1}) \), this coefficient is not in \( (a_0, \ldots, a_{n-1}) \). But the coefficient of \( x^n \) in \( f_{\mathbf{e}'}(x) \) is zero, which lies in \( (a_0, \ldots, a_{n-1}) \). So \( f_{\mathbf{e}}(x) \) and \( f_{\mathbf{e}'}(x) \) cannot be associates.

To see that \( f_{\mathbf{e}}(x) \) is not associate to a polynomial, note that for any unit \( u(x) \in R[[x]]^\times \), the coefficient of \( x^n \) in \( f_{\mathbf{e}}(x) u(x) \) is nonzero whenever \( \epsilon_n = 1 \). Since \( \epsilon_n = 1 \) infinitely often, the product is not a polynomial.
\end{proof}

\vspace{1em}

For the next result, given a commutative ring $R$, we shall say that a power series \( f(x) \in R[[x]] \) is \emph{rational} if there exist polynomials \( p(x), q(x) \in R[x] \), with \( q(x) \) having constant term $1$, such that $f(x)q(x)=p(x)$.\footnote{This definition of rational series is non-standard, and it is often the case that one instead merely insists that the coefficients of $f(x)$ satisfy a non-trivial linear recurrence with coefficients in $R$. A result of Fatou shows that this definition is equivalent to ours when $R$ is a UFD (cf. Stanley \cite[p. 275]{StanleyVol2}).} Since $q(x)$ has constant term $1$, it is a unit in $R[[x]]$ and so a rational series is automatically associate to a polynomial under this definition. A power series that is not rational is called \emph{irrational}.

Equating coefficients of \( f(x) q(x) = p(x) \), we see that \( f(x) = \sum_n a_n x^n \in R[[x]] \) is rational if and only if there exist \( d \geq 1 \) and \( q_0, \ldots, q_d \in R \), with \( q_0=1\) and \( q_d \) nonzero and \( (q_0, \ldots, q_d) = R \), such that
\[
\sum_{i=0}^d a_{n-i} q_i = 0
\]
for all \( n \) sufficiently large.

In particular, if \( \Psi(x) = \sum a_n x^n \in R[[x]] \) has all coefficients in \( \{0,1\} \), then \( \Psi(x) \) is rational if and only if the sequence \( \{a_n\} \) is eventually periodic. Indeed, if \( a_n = a_{n+d} \) for all large \( n \), then \( \Psi(x)(1 - x^d) \in R[x] \). Conversely, if \( \Psi(x) \) is rational, the recurrence relation above and the fact that $a_n\in \{0,1\}$ and $q_0=1$ implies that the tail of \( \{a_n\} \) is determined by a finite amount of data, and since the sequence is \( \{0,1\} \)-valued, it must eventually repeat.

Hence, the set \( \mathcal{Z} \) of irrational power series in \( R[[x]] \) with coefficients in \( \{0,1\} \) is uncountable, since the set of eventually periodic $\{0,1\}$-sequences is countable.

\begin{lemma}\label{prop:PIR}
Let \( R \) be a ring that is not a principal ideal ring. Then \( R[[x]] \) contains an uncountable family of pairwise non-associate power series, none of which are associate to a polynomial.
\end{lemma}

\begin{proof}
If \( R \) is not Noetherian, the result follows from Lemma~\ref{lem:noether}. Otherwise, since \( R \) is not a principal ideal ring, by a result of Kaplansky \cite[Theorem 12.3]{Kaplansky}, \( R \) has a maximal ideal \( P \) that is not principal. By a version of Nakayama's lemma (see \cite[Theorem 2.3]{MatsumuraCommutativeRingTheory}), the dimension of $P/P^2$ is as an $R/P$-vector space is equal to the size of a minimal generating set for $P$ as an ideal, and hence we can choose \( a, b \in P \) such that their images in \( P/P^2 \) are \( R/P \)-linearly independent.

Let \( \mathcal{Z} \) be the (uncountable) set of irrational power series in \( R[[x]] \) with coefficients in \( \{0,1\} \), as above. For each \( \Psi(x) \in \mathcal{Z} \), define \( f_\Psi(x) := a + b \Psi(x) \in R[[x]] \).  We claim that the elements of the uncountable set of power series
$$\{ f_{\Psi}(x)\colon \Psi(x)\in \mathcal{Z}\}$$ are pairwise non-associate.  To see this, suppose that \( \Psi \neq \Phi \in \mathcal{Z} \). If \( f_\Psi(x) \) and \( f_\Phi(x) \) were associate, we would have \( f_\Phi(x) = f_\Psi(x) h(x) \) for some unit \( h(x) \in R[[x]]^\times \), hence
\[
 (a+b\Phi(x)) =(a + b \Psi(x)) h(x).
\]
Equivalently, we have
$a(1-h(x)) = b(\Psi(x) h(x)-\Phi(x))$.
 Working mod \( P^2 \), we get
\[
a (1- h(x)) = b(\Psi(x)h(x) - \Phi(x)) \pmod{P^2}.
\]
Since \( a, b \in P \) and \( a + P^2, b + P^2 \) are linearly independent in \( P/P^2 \), we conclude that \( h(x)-1, \Psi(x)h(x) - \Phi(x) \in P[[x]] \).  In particular, this implies that $\Psi(x)$ and $\Phi(x)$ are equivalent mod $P$, a contradiction, since by construction \( \Psi(x) - \Phi(x) \) has some coefficient in \( \{\pm1\} \).

Now suppose \( f_\Psi(x) \) is associate to a polynomial \( p(x) \in R[x] \). Then \( f_\Psi(x) = p(x) u(x) \) for some unit \( u(x) \in R[[x]]^\times \). Then the coefficients of \( p(x) \) lie in the ideal \( (a,b) \), so write \( p(x) = a p_0(x) + b p_1(x) \). Then modulo \( P^2 \):
\[
a + b \Psi(x) = a p_0(x) u(x) + b p_1(x) u(x).
\]
Rewriting, we obtain:
\[
a (1 - p_0(x) u(x)) = b (p_1(x) u(x) - \Psi(x)) \pmod{P^2}.
\]
Again, by the linear independence of \( a, b \) mod \( P^2 \), both sides must vanish separately:
\[
p_0(x) u(x) \equiv 1 \pmod{P}, \quad p_1(x) u(x) \equiv \Psi(x) \pmod{P}.
\]
In particular, after scaling $p_0$ and $p_1$ by an appropriate unit of $R$ mod $P$ and scaling $u(x)$ by the inverse of this unit mod $P$, we may assume that $p_0$ has constant term that is $1$ mod $P$.

Multiplying the second congruence by \( p_0(x) \), we get:
\[
p_1(x) \equiv p_0(x) \Psi(x) \pmod{P} \]
and so \[ \Psi(x) \equiv \frac{p_1(x)}{p_0(x)} \pmod{P}.
\]
But this implies \( \Psi(x) \) is rational mod \( P \), a contradiction since \( \Psi \in \mathcal{Z} \). Hence \( f_\Psi(x) \) is not associate to a polynomial.
\end{proof}
We will now prove a key proposition that a principal ideal domain (and we include fields among principal ideal domains) that has either the Weierstrass property or its strong variant is complete. To do this, we first require a technical lemma.  We recall that if $L$ is a finite extension of a field $K$ that is complete with respect to a discrete valuation $v$, then $L$ is complete and there is a unique extension of $v$ to $L$ (see Serre \cite[Chapter II, \S2, Prop. 3 \& Cor. 2]{Serre}).
\begin{lemma} Let $R$ be a principal ideal domain that is not a field, let $P$ be a maximal ideal of $R$, let $\widehat{R}$ be the $P$-adic completion of $R$, and let $L$ be a finite extension of ${\rm Frac}(\widehat{R})$.  Suppose that $\beta\in P\widehat{R}$ is nonzero and suppose that $\beta_1,\ldots ,\beta_s\in L\setminus \{\beta\}$ have the property that they each lie in the maximal ideal of the unique valuation ring of $L$ that extends the valuation subring $\widehat{R}$ of ${\rm Frac}(\widehat{R})$. Then there exists $f(x)\in R[[x]]$ such that $f(\beta)=0$ and $f'(\beta),f(\beta_1),\ldots ,f(\beta_s)$ are all nonzero. 
\label{lem:technical}
\end{lemma}
\begin{proof} Let $\pi\in R$ be a generator for $P$.  We pick $d$ such that $\beta\in P^d\widehat{R}\setminus P^{d+1}\widehat{R}$. Since $R$ is dense in $\widehat{R}$, there exists a sequence $\{a_n\}\subseteq R$ with $a_n-\beta\in P^{nd} \widehat{R}$ for all $n\ge 1$. It follows that $a_n - \beta = \beta^n c_n$ for some $c_n\in \widehat{R}$.  Now since $\widehat{R}/(\beta)\cong R/P^d$, we can write $c_n=c_{0,n} + c_{1,n} \beta+ c_{2,n} \beta^2+\cdots$ for some $c_{0,n},c_{1,n},\ldots \in R$.  (We note that the $c_{i,n}$ are not unique.)
We then let $$f_n(x) = a_n -x -\sum_{i\ge 0} c_{0,n} x^{i+n}.$$
By construction $f_n(\beta)=0$, and $f_n'(\beta) = -1 - \sum_{i\ge 0} (i+n) c_{0,n} \beta^{i+n-1}$, which is nonzero for $n\ge 2$, since all terms in the sum $\sum_{i\ge 0} (i+n) c_{0,n} \beta^{i+n-1}$ have $P$-adic absolute value strictly less than $1$ and hence this sum cannot be equal to $-1$ when $n\ge 2$.  

To complete the proof, it suffices to prove that for $n$ large, we have $$f_n(\beta_1)\cdots f_n(\beta_s)\neq 0.$$  Suppose to the contrary that this is not the case.  Then there is some $j\in \{1,\ldots ,s\}$ such that $f_n(\beta_j)=0$ for infinitely many $n$.  We let $|\, \cdot \,|$ be an absolute value on $L$ that gives the topology induced by the unique extension of the $P$-adic absolute value to the complete field $L$. 
Then $|\beta_j|<1$ and so $$\sum_{i\ge 0} c_{0,n} \beta_j^{i+n}\to 0$$ in $L$ as $n\to\infty$.  Then since $f_n(\beta_j)=0$ for infinitely many $n$, we see that $a_n-\beta_j$ has a subsequence tending to $0$ in $L$.  But $a_n$ lies in $R\subseteq {\rm Frac}(\widehat{R})\subseteq L$ and $a_n\to \beta$. This contradicts that $\beta_j\neq \beta$.  It follows that there is some $n\ge 1$ such that taking $f(x)=f_n(x)$ gives the desired power series in $R[[x]]$.
\end{proof}
\begin{proposition}
Let \( R \) be a principal ideal domain. If \( R \) has either the Weierstrass property or the strong Weierstrass property, then \( R \) is a complete discrete valuation ring. 
\label{prop:cdvr}
\end{proposition}

\begin{proof} 
If \( R \) is a field, the result is immediate. So we may assume that \( R \) has Krull dimension one and we let $P$ be a maximal ideal of $R$. By the Krull intersection theorem \cite[Theorem~8.10]{MatsumuraCommutativeRingTheory}, we have \( \bigcap_{n \geq 1} P^n = (0) \), and hence \( R \) embeds into its \( P \)-adic completion \( \widehat{R} \), which is a complete discrete valuation ring.

For any \( a \in P \widehat{R} \), we have a surjective homomorphism \( R[[x]] \to \widehat{R} \) given by evaluation at \( x = a \), since \( \widehat{R}/P^n \widehat{R} \cong R/P^n \). This map is not injective—otherwise \( R[[x]] \cong \widehat{R} \), which is impossible since \( R[[x]] \) is not a PID. Thus, for any \( a \in P \widehat{R} \), there exists a nonzero \( f(x) \in R[[x]] \) such that \( f(a) = 0 \). In particular, this also holds for a generator of $P$.

Since $R$  is a PID, we may divide out any common factor from the coefficients of $f(x)$, ensuring that they generate the unit ideal. Moreover, the constant term of \( f \) must lie in \( P \), since if \( f(x) = c + xg(x) \) and \( f(a) = 0 \), then \( c = -a g(a) \in P \). 

Now, if \( R \) has the Weierstrass property (or the strong version), then \( f(x) = p(x)u(x) \) for some \( p(x) \in R[x] \) and unit \( u(x) \in R[[x]]^\times \). The coefficients of \( p \) must also generate the unit ideal—otherwise, the coefficients of \( f \) would lie in a proper ideal, contradicting our assumption. Evaluating at \( x = a \), we find \( u(a) \) is a unit in \( \widehat{R} \), and since \( f(a) = 0 \), it follows that \( p(a) = 0 \), i.e., \( a \) is algebraic over \( R \). Hence the field of fractions of \( \widehat{R} \) is algebraic over that of \( R \).

We now claim that ${\rm Frac}(\widehat{R})={\rm Frac}(R)$. To see this, it suffices to show that $\widehat{R}\subseteq {\rm Frac}(R)$.  So suppose towards a contradiction that there is some $\beta\in \widehat{R}\setminus {\rm Frac}(R)$. We let $p(x)\in {\rm Frac}(R)[x]$ be the minimal polynomial of $\beta$ and we let $L$ denote the finite extension of ${\rm Frac}(\widehat{R})$ obtained by adjoining the roots of $p(x)$ to ${\rm Frac}(\widehat{R})$.  Then there is a unique valuation on $L$ extending the $P$-adic absolute value on ${\rm Frac}(\widehat{R})$ and $L$ is complete with respect to this valuation \cite[Chapter II, \S2, Prop. 3 \& Cor. 2]{Serre}.  We let $\beta=\beta_0,\beta_1,\ldots ,\beta_s$ denote the distinct roots of $p(x)$ in $L$.  If $\pi\in R$ is a generator for the prime ideal $P$, then there is some $m\ge 1$ such that $p_0(x):=p(\pi^{-m}\cdot x)$ has roots $\pi^m \beta, \pi^m \beta_1,\ldots ,\pi^m \beta_s$, all of which lie in the maximal ideal of the valuation ring of $L$. By Lemma \ref{lem:technical}, 
there exists some $f(x)\in R[[x]]$ such that $f(\pi^m \beta)=0$ and $f'(\pi^m \beta),f(\pi^m \beta_1),\ldots ,f(\pi^m \beta_s)$ are all nonzero.  Now if $R$ has the Weierstrass property or its strong version, then we have $f(x)=q(x)u(x)$ for some unit $u(x)\in R[[x]]$ and some $q(x)\in R[x]$ and since $\pi^m \beta\in P\widehat{R}$ and since $u(x)$ has constant coefficient in $R\setminus P$, we see that $u(\pi^m \beta)\neq 0$ and hence $q(\pi^m\beta)=0$.  Since $p_0(x)$ is the minimal polynomial of $\pi^m \beta$ in ${\rm Frac}(R)[x]$, we see that $q(x)=p_0(x) q_0(x)$ for some polynomial $q_0(x)\in {\rm Frac}(R)[x]$.  In particular, $q(\pi^m \beta_i)=0$ for $i=1,\ldots ,s$, which is impossible unless $s=0$ since $f(\pi^m \beta_i)\neq 0$ for $i=1,\ldots ,s$ and since $u(x)$ is convergent at each $\pi^m \beta_i$.  It follows that $s=0$.  But since $f'(\pi^m \beta)\neq 0$, we see that $p_0(x)$ cannot have $\pi^m \beta$ as a multiple root and so $p_0(x)$ is up to scaling equal to $x-\beta$ and so $\beta\in {\rm Frac}(R)$.

We next show that \( R \) is local. Suppose otherwise: then \( R \) has distinct maximal ideals \( P \) and \( Q \). Since \( P \) and \( Q \) are comaximal, there exist \( u \in P \), \( v \in Q \) such that \( u + v = 1 \), i.e., \( u = 1 - v \). Choose a prime number \( p \) not dividing the characteristic of the residue field of \( R_Q \). Then the binomial series \( (1 - v)^{1/p} \) converges in the \( Q \)-adic completion, and hence \( u \) is a \( p \)-th power in \( \widehat{R} \). Since this holds for all but at most one prime number, Fact~\ref{fact:powerp} implies that \( u \) is a unit, contradicting \( u \in P \). Hence \( R \) is local.

Finally, since \( \widehat{R} \) is faithfully flat over \( R \) \cite[Theorem~8.14]{MatsumuraCommutativeRingTheory}, and \( \mathrm{Frac}(R) = \mathrm{Frac}(\widehat{R}) \), it follows from \cite[p.~53, Ex.~7.2]{MatsumuraCommutativeRingTheory} that \( R = \widehat{R} \). Thus \( R \) is a complete discrete valuation ring.
\end{proof}

\begin{remark}
The question of when the field of fractions of the completion of an integral domain is algebraic over the field of fractions of the original ring has been studied in various contexts. In the local case, such rings are often referred to as \emph{large local rings}, and have been investigated by Zannier and Zanardo \cite{ZanardoZannier}.
Related work on large local rings appears in other studies \cite{Okoh, Piva, Ribenboim, Zanardo}.
\end{remark}

As mentioned in the introduction, complete local principal ideal rings have been classified, and so Theorem \ref{thm:main1}, which we shall now prove, gives a clean description of rings with the strong Weierstrass property.

\begin{proof}[Proof of Theorem \ref{thm:main1}]
Suppose first that \( R \) has the strong Weierstrass property. By Lemma \ref{prop:PIR}, \( R \) must be a principal ideal ring. By a result of Kaplansky \cite[Theorem 12.3]{Kaplansky}, it follows that
\[
R \cong R_1 \times \cdots \times R_m,
\]
where each \( R_i \) is either a principal ideal domain or a local Artinian principal ideal ring.  Moreover, since the class of rings with the strong Weierstrass property is closed under homomorphic images (Lemma \ref{prop:1.3}), it follows that each $R_i$  also has this property.

If \( R_i \) is a principal ideal domain, then by Proposition \ref{prop:cdvr}, \( R_i \) must be a complete discrete valuation ring. Hence, \( R \) is a finite product of complete discrete valuation rings and local Artinian principal ideal rings. This establishes one direction.

For the converse, note that both complete discrete valuation rings and local Artinian principal ideal rings are complete and have the strong Weierstrass property (by Lemma \ref{CLPIPSAP}). Since, by Lemma \ref{prop:1.3}, the class of rings with the strong Weierstrass property is closed under formation of finite direct products, the converse follows.
\end{proof}

\section{The Weierstrass property}
\label{WP}
In this section, we prove Theorem \ref{thm:main2}.  To begin, we need to recall a few basic results.

The following results can be regarded as folklore, although we do not know of proper references and so we provide proofs.
\begin{fact} Let $R$ be a Noetherian integral domain and suppose that $v\in R$ is nonzero and has the property that for infinitely many prime numbers $p$, $v$ is a $p$-th power of an element of ${\rm Frac}(R)$. Then $v$ is a unit of $R$. \label{fact:powerp}
\end{fact}
\begin{proof} Suppose towards a contradiction that $v$ is not a unit and let \( P \) be a minimal prime ideal containing \( v \). By Krull's Principal Ideal Theorem \cite[Corollary 11.17]{AtiyahMacdonald}, \( P \) has height one. Localizing at \( P \), the ring \( R_P \) is a one-dimensional Noetherian local domain, and thus defines a discrete rank-one valuation \( \nu \) on \( \mathrm{Frac}(R) \). Since \( v \in P \), it follows that \( \nu(v) > 0 \).

By assumption, $v$ is a $p$-th power of an element of ${\rm Frac}(R)$ for infinitely many prime numbers $p$, and so \( p \mid \nu(v) \) for infinitely many prime numbers \( p \), contradicting the fact that \( \nu(v) \in \mathbb{Z}_{>0} \), since \( \nu \) is discrete. Therefore, such a \( v \) cannot exist unless \( v \in R^\times \).
\end{proof}
\begin{remark} The hypothesis that the ring be Noetherian is necessary. For example, if $k$ is a field, the direct limit of the nested power series rings $k[[t^{1/n}]]$ has the property that $t$ is an $n$-th power for all $n$, but it is not a unit.
\end{remark}

We next recall a basic fact about power series. While our statement does not provide an explicit formula, it is closely related to the classical Lagrange inversion theorem (see \cite[Chapt. 5]{Stanley}), which gives a closed-form expression for the compositional inverse of a power series with zero constant term in the case when $R$ is an integral domain of characteristic zero.

\begin{fact} \label{fact:comp}
Let $R$ be a (not necessarily commutative) ring. If $f(x) \in x + x^2 R[[x]]$, then there exists a unique power series $g(x) \in x + x^2 R[[x]]$ such that $f(g(x)) = g(f(x)) = x$.
\end{fact}

\begin{proof}
Write $f(x) = x + a_2 x^2 + a_3 x^3 + \cdots$. Consider an arbitrary power series $h(x) = x + \sum_{i \geq 2} b_i x^i$. Taking $a_1 = b_1 = 1$, we compute:

\[
f(h(x)) = x + (b_2 + a_2 b_1^2) x^2 + (b_3 + a_2(b_1 b_2 + b_2 b_1) + a_3 b_1^3) x^3 + \cdots,
\]

where in general, the coefficient of $x^n$ in $f(h(x))$ is given by a polynomial in the $a_i$ and $b_j$ of the form:

\[
b_n + a_2 \sum_{i=1}^{n-1} b_i b_{n-i} + a_3 \sum_{i+j+k=n} b_i b_j b_k + \cdots + a_n b_1^n.
\]

This recursive structure allows us to solve for the coefficients $b_n$ uniquely and inductively: given any $\Theta(x) \in x + x^2 R[[x]]$, we can determine a unique $h(x) \in x + x^2 R[[x]]$ such that $f(h(x)) = \Theta(x)$.

Taking $\Theta(x) = x$, we obtain a unique $g(x)$ such that $f(g(x)) = x$. To see that $g(f(x)) = x$ as well, note that

\[
f(g(f(x))) = (f \circ g)(f(x)) = f(x),
\]

and by applying uniqueness once more, it follows that $g(f(x)) = x$.
\end{proof}
The next result shows that the Weierstrass property has certain closure properties.

\begin{lemma}
The class of rings with the Weierstrass property is closed under formation of arbitrary direct products and taking homomorphic images.
\label{lem:closure}
\end{lemma}

\begin{proof}
To see closure under arbitrary direct products, suppose that $\{R_i\}_{i\in \Gamma}$ is a collection of rings with the Weierstrass property and let $R=\prod R_i$. Suppose that $f(x)=\sum a_n x^n\in R[[x]]$ and that $(a_0,\ldots,a_n)$ generate the unit ideal in $R$. Then for each $i \in \Gamma$, the image $f_i(x):=\sum \pi_i(a_n)x^n \in R_i[[x]]$ satisfies the condition that $\pi_i(a_0),\ldots,\pi_i(a_n)$ generate the unit ideal in $R_i$. Since each $R_i$ has the Weierstrass property, we can write $f_i(x) = p_i(x) u_i(x)$ where $p_i(x) \in R_i[x]$ is monic of degree $n$ and $u_i(x) \in R_i[[x]]^\times$ is a unit.

Using the identification
\[
R[[x]] \cong \prod_{i \in \Gamma} R_i[[x]],
\]
we see that $f(x) = (f_i(x))_{i\in \Gamma} = (p_i(x))_{i\in \Gamma} \cdot (u_i(x))_{i\in \Gamma}$, which gives the desired factorization of $f(x)$ in $R[[x]]$.

The fact that the Weierstrass property is preserved under homomorphic images follows from the same argument used in the proof of Lemma~\ref{prop:1.3}, since passing to a quotient preserves the unit ideal condition and the necessary factorization descends through the homomorphism.
\end{proof}

\begin{proposition}
Let $R$ be a Noetherian integral domain. If $R$ has the Weierstrass property, then $R$ is a complete local ring.
\label{prop:compWP}
\end{proposition}

\begin{proof}
Let $P = (a_1,\ldots,a_m)$ be a maximal ideal of $R$, and let $\widehat{R}$ denote the $P$-adic completion of $R$. The $R$-algebra map $R[t_1,\ldots,t_m] \to R$ given by $t_i \mapsto a_i$ induces a surjection
\[
R[t_1,\ldots,t_m]/(t_1,\ldots,t_m)^i \to R/P^i
\]
for all $i \geq 1$, and taking inverse limits yields a surjective map
\[
\phi: R[[t_1,\ldots,t_m]] \twoheadrightarrow \widehat{R}.
\]
We claim that $R=\widehat{R}$. Assume, for contradiction, that $\widehat{R}$ is not equal to $R$. Then there exists a maximal index $i \in \{0,\ldots,m-1\}$ such that $R$ is equal to
\[
\phi(R[[t_1,\ldots,t_i]]) \subseteq \widehat{R},
\]
but where $\phi(R[[t_1,\ldots,t_i,t_{i+1}]])$ is strictly larger than $R$.

Let $a := a_{i+1}$, and let $S=\phi(R[[t_{i+1}]])=R[[a]]\subseteq \widehat{R}$.  By assumption, $S$ is a subring of $\widehat{R}$ that strictly contains $R$.  We now claim that $a+a^2f(a)\in R$ for all $f(x)\in R[[x]]$.  
To see this, let $f(x) \in R[[x]]$ and let $b := f(a)$. Then, by Fact~\ref{fact:comp}, $x+x^2f(x)$ has a compositional inverse $g(x) \in x + x^2 R[[x]]$ satisfying $g(x+x^2 f(x))=x$ and so $$g(a+a^2b) = a.$$ Set $h(x) := g(x) - a \in R[[x]]$, so $h(x) = -a + x + \cdots$.

Since $R$ has the Weierstrass property, we find a unit $u(x) \in R[[x]]^\times$ and $\lambda \in R$ such that
\[
h(x) = (x - \lambda) u(x).
\]
Then $c: =a+a^2 b$ is a value at which all these series converge, and substituting $x = c$ gives
\[
(c - \lambda) u(c) = h(c)= g(a+a^2b)-a=0,
\]
so $c = \lambda \in R$.  Thus $a+a^2 f(a)\in R$, and so we have obtained the claim; hence $a+a^2 S \subseteq R$. But by construction $S=R+R\cdot a + a^2S\subseteq R$ and so we see that $S=R$, a contradiction.  The result follows.
\end{proof}

\begin{proof}[Proof of Theorem \ref{thm:main2}]
Let $R$ be a Noetherian ring with the Weierstrass property and let $N$ be its nilradical. Since $R$ is Noetherian, $N$ is the intersection of a finite set of minimal prime ideals $P_1,\ldots ,P_m$.  By Lemma \ref{lem:closure}, $R/P_i$ has the Weierstrass property and so $R/P_i$ is a complete local integral domain for $i=1,\ldots ,m$ by Proposition \ref{prop:compWP}.  In particular, $R$ is semilocal, since there is a unique maximal ideal above each $P_i$.  Then \cite[Theorem 5]{Nagata} gives that $R$ is complete with respect to the Jacobson radical of $R$.  Thus $R$ is isomorphic to a finite product of complete Noetherian local rings by \cite[Theorem 8.15]{MatsumuraCommutativeRingTheory}, where each local ring is the completion of $R$ with respect to a maximal ideal.  

Conversely, suppose next that $R$ is isomorphic to a finite direct product of complete Noetherian local rings. Since a complete Noetherian local ring has the Weierstrass property by Theorem \ref{thm:WPT}, and since this property is closed under direct products by Lemma \ref{lem:closure}, we see that $R$ has this property.  
\end{proof}
\begin{remark} \label{rem:strong} Although our choice of terminology suggests that the strong Weierstrass property should be a stronger condition than the ordinary variant, this is not immediate from the definition. Nevertheless, Theorems \ref{thm:main1} and \ref{thm:main2} show that if a ring $R$ has the strong Weierstrass property then it has the usual Weierstrass property too.
\end{remark}
\section{Countable rings}
\label{countable}
Cohen's structure theorem for complete local rings (see \cite[Chapt. 7.8]{Eisenbud} or \cite{Cohen}) implies that
 a complete local integral domain that is not a field is uncountable.  In particular, Theorems \ref{thm:main1} and \ref{thm:main2} show that one can only obtain countable rings with the Weierstrass properties in the Artinian case.  Here we further investigate this phenomenon and show that, in general, countable rings are very far from having these properties in the sense that they may have \emph{uncountably} many associate classes that do not contain a polynomial.

\begin{proposition}\label{thm:5.3}
Let \( R \) be a ring that is not an Artinian principal ideal ring. Then there exists an uncountable set of distinct associate classes in \( R[[x]] \). In particular, if \( R \) is countable and not an Artinian principal ideal ring, then there exists an uncountable set of inequivalent associate classes in \( R[[x]] \), none of which are associate to a polynomial. Moreover, if \( R \) is an Artinian principal ideal ring, then every element of \( R[[x]] \) is associate to a polynomial, so the hypotheses are best possible.
\end{proposition}

\begin{proof}
If \( R \) is not a principal ideal ring, then the result follows from Lemma~\ref{prop:PIR}. Thus we may assume that \( R \) is a principal ideal ring and hence, by \cite[Theorem 12.3]{Kaplansky}, \( R \cong R_1 \times \cdots \times R_m \), where each \( R_i \) is either a PID of Krull dimension one or an Artinian principal ideal ring.

Since a finite product of Artinian principal ideal rings is again Artinian and a principal ideal ring, we may assume without loss of generality that \( R_1 \) is a PID that is not a field. Then it suffices to prove the result when \( R = R_1 \), since if \( \{f_\alpha(x)\} \) is an uncountable set of representatives of distinct associate classes in \( R_1[[x]] \), then the power series \( (f_\alpha(x), 0, \ldots, 0) \) are pairwise non-associate elements in \( R[[x]] \cong (R_1 \times \cdots \times R_m)[[x]] \).

Thus we may assume \( R \) is a PID that is not a field. Then there exists some nonzero \( a \in R \) that is not a unit. Let \( \mathcal{Z} \) denote the uncountable collection of power series of the form
\[
f(x) = a + \sum_{n \ge 1} \epsilon_n x^n,
\]
where \( \{\epsilon_n\} \) is a sequence with \( \epsilon_n \in \{0,1\} \). 

Suppose two distinct power series \( f(x), g(x) \in \mathcal{Z} \) are associate. Then \( f(x) \mid g(x) \), so \( f(x) \mid (g(x) - f(x)) \). But \( g(x) - f(x) \) is a nonzero power series with coefficients in \( \{-1,0,1\} \), and \( f(x) \) has constant term \( a \), which is a non-zero-divisor. Thus \( f(x) \mid (g(x) - f(x)) \) implies \( a \mid 1 \), contradicting the assumption that \( a \) is not a unit. Therefore, the power series in \( \mathcal{Z} \) represent uncountably many distinct associate classes.

This proves the first claim. For the second, observe that if \( R \) is countable, then so is \( R[x] \), so there are only countably many associate classes containing a polynomial. But we have exhibited uncountably many associate classes in \( R[[x]] \), none of which are associate to a polynomial.

Finally, an Artinian principal ideal ring has the strong Weierstrass property. Indeed, by \cite[Theorem 12.3]{Kaplansky}, such a ring is a finite product of local Artinian principal ideal rings. Then by Theorem~\ref{thm:main1}, every power series in \( R[[x]] \) is associate to a polynomial. 
\end{proof}
\begin{remark}
We can think of the above result as saying that a dichotomy holds for countable rings $R$: either every associate class in $R[[x]]$ contains a polynomial or there are uncountably many associate classes not containing a polynomial.  
\end{remark}

\section{A transcendence result}
\label{transcendence}

In this section, we prove a general transcendence result concerning the zeros of \( p \)-adic power series. This will be used to establish Theorem~\ref{thm:main3}. Our key tool will be the use of resultants, which we briefly recall below.

Let \( R \) be an integral domain, and let
\[
f(x) = a_0 x^m + a_1 x^{m-1} + \cdots + a_m, \quad
g(x) = b_0 x^n + b_1 x^{n-1} + \cdots + b_n
\]
be polynomials in \( R[x] \) of degrees \( m \) and \( n \), respectively.

\begin{definition} \label{def:res}
The \emph{resultant} of \( f \) and \( g \), denoted \( \operatorname{Res}(f, g) \), is the determinant of the \emph{Sylvester matrix} \( S(f, g) \), which is the \( (m+n) \times (m+n) \) matrix constructed as follows:

\[
S(f, g) =
\begin{bmatrix}
a_0 & a_1 & \cdots & a_m &        &        &        \\
    & a_0 & a_1    & \cdots & a_m &        &        \\
    &     & \ddots &        &      & \ddots &        \\
    &     &        & a_0 & a_1 & \cdots & a_m \\
b_0 & b_1 & \cdots & b_n &        &        &        \\
    & b_0 & b_1    & \cdots & b_n &        &        \\
    &     & \ddots &        &      & \ddots &        \\
    &     &        & b_0 & b_1 & \cdots & b_n \\
\end{bmatrix}
\]

\noindent
where the first \( n \) rows contain shifted copies of the coefficients of \( f \) padded with zeros to length \( m+n \), and the next \( m \) rows contain similarly shifted copies of the coefficients of \( g \).
\end{definition}

\noindent
The resultant \( \operatorname{Res}(f, g) \in R \) vanishes if and only if \( f \) and \( g \) share a common root in an algebraic closure of the field of fractions of \( R \). We refer the reader to the book of Cox, Little, and O'Shea \cite{CLO} for more details about resultants.

We recall that a map \( b : \mathbb{Z}_{\geq 0} \to \mathbb{Z}_{\geq 1} \) grows \emph{super-exponentially} if \( b(n+1)/b(n) \to \infty \) as \( n \to \infty \). In the following result, we use \( |\, \cdot \,|_p \) for the \( p \)-adic absolute value on \( \mathbb{Q}_p \) and \( |\, \cdot \,| \) for the usual Euclidean absolute value on \( \mathbb{Q} \).

\begin{theorem}
\label{thm:transcendence}
Let \( p \) be a prime number and let \( b(n) \) be an increasing sequence of nonnegative integers with $b(0)=1$ such that $b(n)$ grows super-exponentially with $n$. Suppose that \( a_0, a_1, a_2,\ldots \) are nonzero integers such that:
\begin{enumerate}
\item \( p \mid a_0 \);
\item there exists some \( i \geq 1 \) such that \( p \nmid a_i \);
\item there are constants \( C, \kappa > 1 \) such that \( |a_n| \leq \kappa C^{b(n)} \) for all \( n \).
\end{enumerate}
Then the power series
\[
f(x) := \sum_{n \geq 0} a_n x^{b(n)}
\]
has a root \( \lambda \in \{ z \in \overline{\mathbb{Q}}_p : |z|_p < 1 \} \) that is transcendental over $\mathbb{Q}$.
\end{theorem}

\begin{proof}
We first note that if we regard $f$ as a power series in $\mathbb{Z}_p[[x]]$ then if $m$ is the smallest index $i$ for which $p\nmid a_i$, then by the Weierstrass preparation theorem there is a monic polynomial $Q(x)\in \mathbb{Z}_p[x]$ of degree $m$ and a unit $u(x)\in \mathbb{Z}_p[[x]]$ such that $Q(x)=f(x)u(x)$.  Then the product of the roots of $Q(x)$ is equal to the constant coefficient of $(-1)^m f(x)u(x)$, which is in $p\mathbb{Z}_p\setminus \{0\}$. Hence there is some nonzero root $\lambda$ of $Q(x)$ in $\overline{\mathbb{Q}}_p$ that has $p$-adic absolute value strictly less than $1$, and since $u(x)$ is a unit, we see that $\lambda$ is a root of $f(x)$.  

Now suppose for contradiction that \( \lambda \) is algebraic over \( \mathbb{Q} \). Then there exists an irreducible polynomial
\[
P(x) = p_0 + p_1 x + \cdots + p_d x^d \in \mathbb{Z}[x]
\]
such that \( P(\lambda) = 0 \).

For each \( N \in \mathbb{Z}_{>0} \), define the truncation
\[
\Phi_N(x) := a_0 + a_1 x^{b(1)} + \cdots + a_N x^{b(N)}.
\]
We claim that $\Phi_N(\lambda)$ is nonzero for infinitely many $N$.  To see this, observe that if this is not the case then we must have $\Phi_n(\lambda)=0$ for all sufficiently large $n$.  But then 
$a_{n+1}\lambda^{b(n+1)}=\Phi_{n+1}(\lambda)-\Phi_n(\lambda)=0$ for all sufficiently large $n$. This is impossible, however, as the $a_i$ are nonzero and $\lambda$ is nonzero.

If \( N \) is such that \( \Phi_N(\lambda) \neq 0 \), then
\begin{equation}
\label{eq:bound}
0 < |\Phi_N(\lambda)|_p = \left| f(\lambda) - \Phi_N(\lambda) \right|_p = \left| \sum_{n > N} a_n \lambda^{b(n)} \right|_p 
\end{equation}
We claim that for $n$ sufficiently large we have $|a_n \lambda^{b(n)}|_p > |a_{n+1} \lambda^{b(n+1)}|_p$ and hence if we choose $N$ large enough such that $\Phi_N(\lambda) \neq 0$ then Equation (\ref{eq:bound}) gives
\begin{equation}
\label{eq:bound'}
0 < |\Phi_N(\lambda)|_p \le |a_{N+1}|_p |\lambda|_p^{b(N+1)}.
\end{equation}
To show this claim, note that since $a_{n+1}$ is an integer, we have
$ |a_{n+1} \lambda^{b(n+1)}|_p \le |\lambda|_p^{b(n+1)}$, and since $|a_n| \le \kappa C^{b(n)}$ we have $|a_n|_p \ge \kappa^{-1} C^{-b(n)}$.  Thus to prove the claim, it suffices to show that 
$|\lambda|_p^{b(n+1)-b(n)} < \kappa^{-1} C^{-b(n)}$ for $n$ sufficiently large. But by assumption the sequence $b(n)$ grows super-exponentially with $n$ and so $b(n+1)/b(n)\to \infty$.  Then since $|\lambda|_p<1$, taking logarithms and comparing gives the claim.  Thus we have established Equation (\ref{eq:bound'})

Since \( P(x) \) is irreducible and \( \Phi_N(\lambda) \neq 0 \), the polynomials \( P(x) \) and \( \Phi_N(x) \) are coprime when regarded as polynomials in $\mathbb{Q}[x]$. Thus their resultant \( B_N := \operatorname{Res}(P, \Phi_N) \) is a nonzero integer \cite[Chapt. 3, \S6, Proposition 3]{CLO}. Moreover, by \cite[Chapt. 3, \S6, Proposition 5]{CLO}, there exist integer polynomials \( U_N(x), V_N(x) \) such that
\begin{equation}
\label{eq:Bezout}
U_N(x) P(x) + V_N(x) \Phi_N(x) = B_N.
\end{equation}
The Sylvester matrix used to compute \( B_N \) has \( b(N) + d \) rows and columns. Using Hadamard's inequality (see \cite[Corollary~7.8.3]{HornJohnson}), we estimate:
\[
|B_N|^2 \leq \left( \sum_{i=0}^d p_i^2 \right)^{b(N)} \cdot \left( \sum_{i=0}^N a_i^2 \right)^d.
\]
Since \( |a_i| \leq \kappa C^{b(i)} \), we obtain
\[
\sum_{i=0}^N a_i^2 \leq \sum_{i=0}^N \kappa^2 C^{2b(i)} \le \kappa^2 (1+C^2+C^4+\cdots + C^{2b(N)}),
\]
where the right-most inequality follows from the fact that the $b(i)$ are increasing.  Notice that we can bound the right side by
$$\kappa^2 C^{2b(N)}(1+C^{-2}+C^{-4}+\cdots ) = \kappa^2 C^{2b(N)} (1-C^{-2})^{-1}.$$
Letting \( L := \sum_{i=0}^d p_i^2 \), we have
\begin{equation} \label{eq:Bn1}
0 < |B_N|^2 \leq \kappa^2 (1 - C^{-2})^{-1} L^{b(N)} C^{2b(N)}.
\end{equation}

Now, substituting \( x = \lambda \) into Equation~\eqref{eq:Bezout} gives
\[
B_N = V_N(\lambda) \Phi_N(\lambda),
\]
and since \( V_N \) has integer coefficients and \( |\lambda|_p < 1 \), we have \( |V_N(\lambda)|_p \leq 1 \), so
\[
0 < |B_N|_p \leq |\Phi_N(\lambda)|_p < |a_{N+1}|_p \cdot |\lambda|_p^{b(N+1)} \le |\lambda|_p^{b(N+1)}.
\]
Letting \( c := |\lambda|_p < 1 \), we then get
\[
|B_N|^2 \geq |B_N|_p^{-2} > c^{-2b(N+1)}
\]
Combining this with the earlier upper bound from Equation (\ref{eq:Bn1}) yields
\[
\kappa^{2} (1 - C^{-2})^{-1} L^{b(N)} C^{2b(N)} > c^{-2b(N+1)}.
\]
Taking logarithms and dividing by \( b(N) \), we obtain
\[
\log L + 2\log C + \frac{1}{b(N)} \log\left( \kappa^{2}(1 - C^{-2})^{-1} \right) > -\frac{2b(N+1)}{b(N)} \log c.
\]
Since \( -\log c > 0 \) and \( b(N+1)/b(N) \to \infty \), the right-hand side tends to \( \infty \), while the left-hand side is bounded above. This is a contradiction, so \( \lambda \) cannot be algebraic. 
\end{proof}

We also have a positive characteristic analogue of this result.  To simplify our arguments we use the non-Archimedean absolute value on $\mathbb{F}_p[[t]]$ defined by $|0|=0$ and $|t^n u(t)|=1/2^n$ whenever $u(t)$ is a power series with nonzero constant term. The algebraic closure of the field $\mathbb{F}_p((t))$ is somewhat more complicated than that of Laurent series over a characteristic zero field due to the existence of Artin-Schreier extensions and inseparability phenomena. One can nevertheless give a concrete description of an algebraically closed extension of $\mathbb{F}_p((t))$ (which is strictly larger than the actual algebraic closure) in terms of Hahn power series (see Kedlaya \cite{Kedlaya1, Kedlaya}). Here, given a field $K$, we let $K((t^{\mathbb{Q}}))$ denote the formal series of the form
\begin{equation}
\label{eq:Hahn}
\sum_{\alpha\in \mathbb{Q}} c_{\alpha} t^{\alpha},
\end{equation} where $c_{\alpha}\in K$ for $\alpha\in \mathbb{Q}$ and where the set $\{\alpha\in \mathbb{Q}\colon c_{\alpha}\neq 0\}\subseteq \mathbb{Q}$ is well-ordered. This well-ordered constraint allows one to multiply formal series in this ring \cite[\S2]{Kedlaya1}.  Then if $K$ is algebraically closed then $K((t^{\mathbb{Q}}))$ is itself algebraically closed \cite[Proposition 1]{Kedlaya1} and we can extend the usual valuation on $K((t))$ to $K((t^{\mathbb{Q}}))$ by declaring that the valuation of a nonzero series of the form given in Equation (\ref{eq:Hahn}) is $\min\{ \alpha\colon c_{\alpha}\neq 0\}$. We then let $K((t^{\mathbb{Q}}))_{>0}$ denote the set of Hahn power series with positive valuation, which is the maximal ideal of the valuation subring of this field.

\begin{theorem}
\label{thm:transcendencep}
Let $p$ be a prime number and let $b(n)$ be an increasing sequence of nonnegative integers that grows super-exponentially with $n$ such that $b(0)=1$.  Suppose that 
$a_0, a_1, a_2,\ldots $ are nonzero polynomials in $\mathbb{F}_p[t]$ such that:
\begin{enumerate}
\item $a_0\in t\mathbb{F}_p[t]$;
\item $a_i\not\in t\mathbb{F}_p[t]$ for some $i>0$;
\item there is a positive constant $C$ such that ${\rm deg}(a_n)\le C\cdot b(n)$ for all $n\ge 1$.
\end{enumerate}
Then the power series $$f(x):=\sum_{n\ge 0} a_n x^{b(n)}$$ has a root $\lambda$ in $\overline{\mathbb{F}}_p((t^{\mathbb{Q}}))_{>0}$ with $\lambda$ not algebraic over $\mathbb{F}_p(t)$.
\end{theorem}
\begin{proof} We argue as in the proof of Theorem \ref{thm:transcendence}. 
We first note that if we regard $f$ as a power series in $\mathbb{F}_p[[t]][[x]]$ then if we let $m$ denote the smallest index $i$ for which $t\nmid a_i$, by the Weierstrass preparation theorem there is a monic polynomial $Q(x)\in \mathbb{F}_p[[t]][x]$ of degree $m$ and a unit $u(x)\in \mathbb{F}_p[[t]][[x]]$ such that $Q(x)=f(x)u(x)$.  Then the product of the roots of $Q(x)$ is equal to the constant coefficient of $(-1)^m f(x)u(x)$, which is in $t\mathbb{F}_p[t]\setminus \{0\}$. Hence there is some nonzero root $\lambda$ of $Q(x)$ in $\overline{\mathbb{F}}_p((t^{\mathbb{Q}}))_{>0}$, and since $u(x)$ is a unit, we see that $\lambda$ is a root of $f(x)$.  

We let 
$\Phi_N(x) = \sum_{n=0}^N a_n x^{b(n)}$.  Then since $\Phi_{n+1}(\lambda)-\Phi_n(\lambda) =a_{n+1}\lambda^{b(n+1)}\ne 0$, we see that for infinitely many $N$ we must have $0<|\Phi_N(\lambda)| = 1/2^{b(N+1)}$.  On the other hand if $\lambda$ is algebraic then there is some polynomial $P\in \mathbb{F}_p[t][x]$ such that $P(\lambda)=0$.  We can again compute the resultant of $P$ and $\Phi_N$ and we see by 
Definition \ref{def:res} that if $H$ denotes the maximum of the $t$-degrees of the coefficients of $P$, then the resultant is 
the determinant of a $(b(N)+{\rm deg}(P))\times (b(N)+{\rm deg}(P))$ matrix in which the first ${\rm deg}(P)$ rows have entries in $\{0, a_0,\ldots ,a_N\}$ and $b(N)$ rows have entries that are either coefficients of $P$ or $0$.  In particular, since the determinant is a product of at most one entry from each row, we see that the determinant is a polynomial in $t$ that has degree at most
$$H \cdot b(N) + \max({\rm deg}_t(a_0),\ldots ,{\rm deg}_t(a_N))\cdot {\rm deg}(P).$$  By assumption this degree is $O(b(N))$ and the resultant is nonzero \cite[Chapt. 3, \S6, Proposition 3]{CLO}, provided $\Phi_N(\lambda)\neq 0$.  On the other hand, 
by \cite[Chapt. 3, \S6, Proposition 5]{CLO} we have that there exist polynomials $U_N(x)$ and $V_N(x)$ in $\mathbb{F}_p[t][x]$ such that the resultant is equal to $$f(x) U_N(x) + P(x) V_N(x).$$ Substituting $x=\lambda$ and using the fact that
$|f(\lambda)|,|\Phi_N(\lambda)| \le 1/2^{b(N+1)}$, $|\lambda|<1$, we see that 
$$0<|f(\lambda) U_N(\lambda)+P(\lambda) V_N(\lambda)| \le 1/2^{b(N+1)}$$ whenever $\Phi_N(\lambda)\neq 0$ and so
the degree of the resultant must be at least $b(N+1)$ for infinitely many $N$.  Thus we obtain a contradiction.  The result follows.
\end{proof}

The connection between this transcendence result and a power series not being associate to a polynomial is a consequence of the following result.  We recall that in a Noetherian integral domain with maximal ideal $P$, the intersection of the powers of $P$ is zero by the Krull intersection theorem \cite[Theorem~8.10]{MatsumuraCommutativeRingTheory} and so $R$ can be regarded as a subring of the completion of $R$ with respect to the maximal ideal $P$.
\begin{proposition}
\label{prop:alg}
Let $R$ be a Noetherian integral domain with maximal ideal $P$ and let $\widehat{R}$ denote the completion of $R$ with respect to $P$.  Suppose that $f(x)\in R[[x]]$ has constant coefficient in $P$, and that, when regarded as a power series over $\widehat{R}$, $f(x)$ has a zero $\alpha$ that is not algebraic over $R$ and such that $\alpha^m \in P\widehat{R}$ for some $m\ge 1$. Then $f(x)$ is not associate to a polynomial.
\end{proposition}
\begin{proof} If we have a factorization $f(x)=p(x)u(x)$ with $p(x)\in R[x]$ and $u(x)$ a unit of $R[[x]]$, then $p(x)=f(x)u(x)^{-1}$ and both $f(x)$ and $u(x)^{-1}$ are convergent power series on the maximal ideal of the valuation ring of some algebraic extension of the field of fractions of $\widehat{R}$.  Thus if $\alpha$ is a zero of $f(x)$ such that $\alpha^m$ is in $P\widehat{R}$ for some $m\ge 1$ then $p(\alpha)$ is necessarily equal to zero and so $\alpha$ is algebraic over $R$.
\end{proof}
We are now able to prove a general result, which when combined with Theorem \ref{thm:transcendence} gives a strengthened version of Theorem \ref{thm:main3} and so in particular constructs explicit examples of power series that are not associate to polynomials. We note that if $R$ is a finitely generated integral domain, then its group of units is a finitely generated abelian group by a result of Roquette \cite{Roquette} and so if $R$ has characteristic zero then only finitely many integer primes are units of $R$.   

Before giving the proof of Theorem \ref{thm:main3}, we require a general fact about algebraic extensions of completions. The following can be regarded as folklore, although we are unaware of an explicit reference in the generality that we need, so we give a brief sketch.
\begin{fact} Let $R$ be a Noetherian discrete valuation ring and let $\widehat{R}$ and $K$ denote respectively the completion of $R$ and the field of fractions of $R$.  If $L$ is a finite extension of ${\rm Frac}(\widehat{R})$ endowed with a prolongation of the valuation on ${\rm Frac}(\widehat{R})$ then there is a finite extension $L_0$ of $K$ and an absolute value on $L_0$ such that $L$ is isomorphic to a purely inseparable extension of the completion of $L_0$ with respect to the extended valuation.  In particular, in characteristic zero we have $L$ is isomorphic to the completion of $L_0$ as valued fields.
\label{Krasner}
\end{fact}
\begin{proof} (Sketch) Let $L'$ denote the separable closure of ${\rm Frac}(\widehat{R})$ in $L$.  Then it suffices to show that there is a finite extension $L_0$ of $K$ and an absolute value on $L_0$ such that $L'$ is isomorphic to the completion of $L_0$. 
Then since $L'$ is separable over ${\rm Frac}(\widehat{R})$, there is an element $\alpha$ such that
$$L'={\rm Frac}(\widehat{R})(\alpha).$$ Let $Q(x)\in {\rm Frac}(\widehat{R})[x]$ denote the minimal polynomial of $\alpha$ and let $E$ denote a splitting field of $Q(x)$ and let $|\,\cdot \,|_E$ denote an extension of the absolute value on $L_0$ to $E$. Using density of $R$ in $\widehat{R}$, we see there is a polynomial $Q_0(x)\in K[x]$ such that $Q_0(x)$ has a root $\beta$ that is closer to $\alpha$ with respect to a prolongation of $|\, \cdot \,|_E$ than to all other roots of $Q(x)$. Then by Krasner's lemma \cite[Corollaire 1, p. 190]{Ribenboim2}, we have $L'={\rm Frac}(\widehat{R})(\beta)$. Letting $L_0=K(\beta)$ now gives the result. 
\end{proof}
\begin{proof}[Proof of Theorem \ref{thm:main3}]
The first part of the result follows from Theorem \ref{thm:transcendence}.  Thus $f(x)$ has a root $\lambda\in \overline{\mathbb{Q}}_p$ that is not algebraic over $\mathbb{Q}$ and which lies in the maximal ideal of the integral closure of $\mathbb{Z}_p$ in $\overline{\mathbb{Q}}_p$.  Thus we can pick a finite extension $K$ of $\mathbb{Q}$ such that $\lambda$ lies in the maximal ideal of the completion of $\mathcal{O}_K$, where $\mathcal{O}_K$ is the valuation ring of $K$ with respect to a prolongation of the $p$-adic absolute value to $K$. Fact \ref{Krasner} gives that $K$ can be regarded as the completion of a number field $K_0$ with respect to a suitable prolongation of the $p$-adic absolute value to $K_0$.  Letting $R$ denote the valuation ring of $K_0$, we see that $\widehat{R}=\mathcal{O}_K$.
The fact that $f(x)$ is not associate to a polynomial now follows from Proposition \ref{prop:alg}, regarding $f(x)$ as a power series in $R[[x]]$ and observing that by construction the constant coefficient of $f(x)$ is in the maximal ideal of the local ring $R$ and that $f(x)$ has a root in the maximal ideal of $\widehat{R}=\mathcal{O}_K$ that is not algebraic over $K_0$.
\end{proof}
We note that an analogue of Theorem \ref{thm:main3} holds for the ring $\mathbb{F}_p[t]$, which we now state and prove.

\begin{theorem} \label{rem:nonassociate} 
Let $p$ be a prime number and let $b(n)$ be an increasing sequence of nonnegative integers with $b(0)=1$ such that $b(n+1)/b(n)\to\infty$ as $n\to\infty$.  Suppose that 
$a_0, a_1, a_2,\ldots $ are nonzero polynomials in $\mathbb{F}_p[t]$ such that:
\begin{enumerate}
\item $a_0\in t\mathbb{F}_p[t]$;
\item $a_i\not\in t\mathbb{F}_p[t]$ for some $i>0$;
\item there is a positive constant $C$ such that ${\rm deg}(a_n)\le C\cdot b(n)$ for all $n\ge 1$.
\end{enumerate}
Then the power series $$f(x):=\sum_{n\ge 0} a_n x^{b(n)}$$ has a root $\lambda$ in $\overline{\mathbb{F}}_p((t^{\mathbb{Q}}))_{>0}$ with $\lambda$ not algebraic over $\mathbb{F}_p(t)$. Moreover, $f(x)$ cannot be factored in $\mathbb{F}_p[t][[x]]$ as a product of a polynomial and a unit of $\mathbb{F}_p[t][[x]]$, which demonstrates the failure of the Weierstrass property in this setting.
\end{theorem}
\begin{proof}
The existence of such a $\lambda$ with the desired properties follows from Theorem \ref{thm:transcendencep}. Thus we have $\lambda\in \overline{\mathbb{F}}_p((t^{\mathbb{Q}}))_{>0}$ with $\lambda$ not algebraic over $\mathbb{F}_p(t)$.  We can thus pick a finite extension $K$ of $\mathbb{F}_p((t))$ such that $\lambda$ lies in the maximal ideal of the completion of $\mathcal{O}_K$, where $\mathcal{O}_K$ is the valuation ring of $K$ with respect to a prolongation of the $t$-adic absolute value to $K$. Fact \ref{Krasner} implies that the separable closure $K'$ of $\mathbb{F}_p((t))$ in the field $K$ can be regarded as the completion of a finite extension $K_0$ of $\mathbb{F}_p(t)$ with respect to a suitable prolongation of the $t$-adic absolute value to $K_0$.  Letting $R$ denote the valuation ring of $K_0$, we see that $\widehat{R}=\mathcal{O}_{K'}$.
The fact that $f(x)$ is not associate to a polynomial now follows from Proposition \ref{prop:alg}, regarding $f(x)$ as a power series in $R[[x]]$ and observing that by construction the constant coefficient of $f(x)$ is in the maximal ideal of the local ring $R$ and that $f(x)$ has a root $\lambda$ that is not algebraic over $K_0$ and such that $\lambda^{p^i}$ lies in the maximal ideal of $\widehat{R}=\mathcal{O}_K$ for some $i\ge 0$.
\end{proof}
\section{Undecidability}
\label{undecidable}

Among the famous list of problems proposed by David Hilbert in his 1900 ICM address is the \emph{tenth problem}, which asks whether there exists an algorithm that takes an integer polynomial \( P(x_1,\ldots,x_n) \) as input and determines whether the Diophantine equation \( P(x_1,\ldots,x_n)=0 \) has an integer solution. This was ultimately resolved in the negative by Matiyasevich \cite{Mat}, who showed that no such algorithm can exist.\footnote{This result is often called the ``DPRM theorem,'' as Matiyasevich’s argument builds on earlier work of Davis, Putnam, and Robinson \cite{DPR}.} 

To state this more precisely, we must clarify what is meant by an ``algorithm.'' In this context, an algorithm is a Turing machine which takes as input a polynomial \( P(x_1,\ldots,x_n) \) (encoded in some standard way) and halts after a finite number of steps, outputting either \emph{yes} or \emph{no}, depending on whether \( P=0 \) has an integer solution. Informally, one can think of this as asking whether a computer program exists which will correctly answer the question in all cases (regardless of how long it may take). The Church-Turing thesis asserts that any function that is effectively computable by a mechanical process can be computed by a Turing machine.

Just as in the case of Hilbert's Tenth Problem, one can ask about the decidability of many mathematical questions, and similar undecidability results appear in various domains, such as the word problem for groups \cite{Novikov}, the domino problem \cite{Berger}, the Post correspondence theorem \cite{Post}, and variants of the Collatz conjecture \cite{Conway}.

In our setting, it is natural to ask whether there exists a decision procedure for determining whether a power series \( f(x) \in \mathbb{Z}[[x]] \) is associate to a polynomial. Of course, this question is not meaningful in full generality, since arbitrary power series have infinitely many coefficients and a Turing machine can only take finite data as input. To restrict to a computable setting, we consider only those power series whose coefficients are given by a computable function—that is, we assume there exists a Turing machine that, given an input \( n \), computes the coefficient of \( x^n \) in \( f(x) \).

\begin{definition}
A power series \( f(x) = \sum_{n \ge 0} a(n) x^n \in \mathbb{Z}[[x]] \) is called \emph{computable} if there exists a Turing machine that computes the coefficient function \( a : \mathbb{N} \to \mathbb{Z} \).
\end{definition}
We note that this notion of a computable power series naturally extends to any ring \( R \) for which some set of its elements can be encoded using a (computable) language over a finite alphabet—much like we can encode integers via pairs \( (\epsilon, w) \), where \( \epsilon \in \{\pm\} \) indicates the sign and \( w \) is a binary string representing the magnitude.

In this restricted setting, we can now show that the problem of deciding whether a computable power series is associate to a polynomial is undecidable. As shown in Proposition~\ref{prop:alg}, the question of whether a power series is associate to a polynomial is closely tied to the algebraicity of its zeros in completions of the base ring. In the case of integer power series, this is tied to the problem of transcendence of \( p \)-adic zeros. Proving transcendence of such zeros is notoriously difficult. For instance, even in the classical complex-analytic setting, the transcendence of \( \pi \), which is a non-trivial zero of \( \sin x \), is a deep result. It is therefore perhaps unsurprising that no algorithm can determine whether a general computable power series in \( \mathbb{Z}[[x]] \) is associate to a polynomial.

\begin{proof}[Proof of Theorem \ref{thm:undecidable}]
Given a polynomial \( P(x_1,\ldots,x_d) \in \mathbb{Z}[x_1,\ldots,x_d] \), we construct a computable function \( a_P : \mathbb{N} \to \mathbb{Z} \) as follows. First, fix a computable bijection \( \theta : \mathbb{N} \to \mathbb{Z}^d \). (Figure~\ref{fig:comp} shows such a bijection in the case \( d = 2 \).)

Now define a function \( b_P : \mathbb{N} \to \mathbb{N} \) recursively by:
\[
b_P(0) = 1, \qquad
b_P(n) = b_P(n-1) + b_P(n-1)^{\prod_{\ell=0}^n P(\theta(\ell))^2 \cdot \prod_{\ell=0}^n (1 + P(\theta(\ell))^2)} \quad \text{for } n > 0.
\]
If \( P \) has an integer zero, then there exists \( j \) such that \( P(\theta(j)) = 0 \), in which case all exponents vanish for \( n \ge j \), and hence \( b_P(n) = b_P(n-1) + 1 \) for all large \( n \). On the other hand, if \( P \) has no integer zeros, then
\[
\prod_{\ell=0}^n P(\theta(\ell))^2 \cdot \prod_{\ell=0}^n (1 + P(\theta(\ell))^2) \ge 2^{n+1},
\]
and so \( b_P(n) \ge b_P(n-1) + b_P(n-1)^{2^n} \), and \( b_P(n) \) grows super-exponentially.

We shall now define our computable power series in $T[[x]]$.  Since $T$ is infinite and finitely generated, it is not Artinian.  Then it has some prime ideal $P$ such that $S:=T/P$ is a finitely generated integral domain of Krull dimension one.  Then either $S$ has characteristic zero, in which case ${\rm Frac}(S)$ is a finite extension of $\mathbb{Q}$, or it has positive characteristic and is a finite module over a polynomial subring $\mathbb{F}_p[t]$ for some prime $p$ by Noether normalization. In the first case, by Roquette's theorem \cite{Roquette}, the group of units is finitely generated and so there is an integer prime $p$ that is not a unit of $S$ and we take $a_P(0)=p$; otherwise, $S$ is a finite module over a subring $\mathbb{F}_p[t]$ and by going up, we have that $t$ is not a unit in $S$ and we take $a_P(0)=t'$, where $t'\in T$ is an element whose image in $S$ is equal to $t$.

We now define the coefficient function for $n>0$:
\[
a_P(n) = 
\begin{cases}
1 & \text{if } n = b_P(j) \text{ for some } j > 0, \\
0 & \text{otherwise}.
\end{cases}
\]
This function is computable since the sequence $b_P(j)$ is increasing with $j$. Define the power series
\[
f_P(x) := \sum_{n \ge 0} a_P(n) x^n = a_P(0) + \sum_{j \ge 0} x^{b_P(j)}.
\]

If \( P \) has an integer zero, then \( b_P(j) = 1 + b_P(j-1) \) for large \( j \), so \( f_P(x) \) is equal to $q(x)(1-x)^{-1}$ for some polynomial $q(x)$ and hence is associate to a polynomial. On the other hand, if \( P \) has no integer zeros, then we claim that $f_P(x)$ is not associate to a polynomial.  To see this, we note that it suffices to show that the image of $f_P(x)$ in $S[[x]]$ is not associate to a polynomial. Since \( b_P(n) \) grows super-exponentially, we claim that the image of \( f_P(x) \) in $S[[x]]$ is not associate to a polynomial.  To see this, observe that in the case when $S$ has characteristic zero, the field of fractions of $S$ is a finite extension of $\mathbb{Q}$, and by Theorem~\ref{thm:main3}, $f_P(x)$ has a root in $\overline{\mathbb{Q}}_p$ of $p$-adic absolute value strictly less than one that is not algebraic over $\mathbb{Q}$, and hence it is not algebraic over ${\rm Frac}(S)$. 
Fact \ref{Krasner} now gives that there is a finite extension $K$ of ${\rm Frac}(S)$ such that $\lambda$ lies in the maximal ideal of the valuation ring of the completion of $K$ with respect to an extension of the $p$-adic absolute value to $K$. Taking $R$ to be the valuation ring of $K$, we see that $\lambda$ lies in the maximal ideal of the completion of $R$ and $f_P(x)$ has constant coefficient in the maximal ideal of $R$.  Since $\lambda$ is not algebraic over $R$,
by Proposition \ref{prop:alg} we then see that $f_P(x)$ is not associate to a polynomial when regarded as a power series in $S[[x]]$ and hence is not associate to a polynomial as an element in $T[[x]]$. 

The positive characteristic case is handled similarly, noting that if $P$ has an integer zero then $f_P(x)$ is again associate to a polynomial.  To complete the proof in the positive characteristic case, we must show that if $P$ does not have an integer zero then $f_P(x)$ is not associate to a polynomial.  By construction we can identify $\mathbb{F}_p[t]$ with a subring of $S$ such that $S$ is a finite extension of $\mathbb{F}_p[t]$ and where $f_P(x)\in \mathbb{F}_p[t][[x]]$. Also by Theorem \ref{rem:nonassociate}, $f_P(x)$ has a root $\lambda$ in $\overline{\mathbb{F}}_p((t^{\mathbb{Q}}))_{>0}$ that is not algebraic over $\mathbb{F}_p(t)$ (and hence not over the field of fractions of $S$).  

We let $K\subseteq \overline{\mathbb{F}}_p((t^{\mathbb{Q}}))_{>0}$ be a finite extension of $\mathbb{F}_p((t))$ such that the maximal ideal of the valuation subring of $K$ (with respect to the $t$-adic valuation) contains $\lambda$. Then by Fact \ref{Krasner} there is a finite extension $K_0$ of ${\rm Frac}(S)$ such that if $R$ is the valuation subring of $K_0$ (with respect to a prolongation of the $t$-adic absolute value of $\mathbb{F}_p(t)$ to $K_0$) then $K$ is purely inseparable over the field of fractions of $\widehat{R}$. In particular, since $\lambda$ has $t$-adic absolute value $<1$, $\lambda^{p^i}$ is in the maximal ideal of $\widehat{R}$. Then using Proposition \ref{prop:alg} we see that $f_P(x)$ is not associate to a polynomial when regarded as an element of $R[[x]]$ and hence is not associate to a polynomial when regarded as an element of $S[[x]]$ and $T[[x]]$. 

Thus the undecidability of determining whether \( f_P(x) \) is associate to a polynomial follows from the undecidability of Hilbert’s tenth problem.
\end{proof}

\begin{figure}[t]
\centering
\begin{tikzpicture}[scale=1.2]
  \newcounter{n}
  \setcounter{n}{0}

  \foreach \x/\y in {
    0/0, 1/0, 1/1, 0/1, -1/1, -1/0, -1/-1, 0/-1, 1/-1,
    2/-1, 2/0, 2/1, 2/2, 1/2, 0/2, -1/2, -2/2, -2/1, -2/0,
    -2/-1, -2/-2, -1/-2, 0/-2, 1/-2, 2/-2
  } {
    \draw[gray!50] (\x,\y) rectangle +(1,1);
    \node at (\x+0.5,\y+0.5) {\scriptsize $\theta(\then)$};
    \stepcounter{n}
  }

  \foreach \x in {-2,...,2} {
    \node at (\x+0.5,-2.7) {\small $i = \x$};
  }
  \foreach \y in {-2,...,2} {
    \node at (-2.7,\y+0.5) {\small $j = \y$};
  }
\end{tikzpicture}
\caption{A computable bijection \( \theta : \mathbb{N} \to \mathbb{Z}^2 \) via a spiral walk.}
\label{fig:comp}
\end{figure}

\begin{remark}
The hypothesis that $T$ be infinite cannot be removed completely, as an Artinian principal ideal ring has the property that every power series over $T$ is associate to a polynomial by using a result of Kaplansky \cite[Theorem 12.3]{Kaplansky} and Theorem \ref{thm:main1}.  The hypothesis that $T$ be finitely generated as a ring also cannot be removed completely in the statement of Theorem \ref{thm:undecidable}; for example if $T$ is the (not finitely generated) ring $\mathbb{Q}$, then every nonzero element of $T[[x]]$ is expressible in the form $x^m u(x)$ for some $m\ge 0$ and some unit $u(x)$ and so the decision problem is trivially decidable in this case. \end{remark}
\section*{Acknowledgments}
We thank Rahim Moosa for important comments.  We are also indebted to the referee for numerous helpful comments and suggestions.

\end{document}